\documentclass[12pt, reqno]{amsart}
\usepackage{amsmath, amstext, amsbsy, amssymb, color}
\usepackage[mathscr]{eucal}
\usepackage[linktocpage=true]{hyperref}

\newtheorem{theorem}{Theorem}[section]

\newtheorem{Definition-Lemma}[theorem]{Definition-Lemma}
\newtheorem{lemma}[theorem]{Lemma}
\newtheorem{proposition}[theorem]{Proposition}

\newtheorem{corollary}[theorem]{Corollary}

\theoremstyle{definition}
\newtheorem{definition}[theorem]{Definition}

\newtheorem{conjecture}[theorem]{Conjecture}

\newtheorem{example}[theorem]{Example}

\theoremstyle{remark}
\newtheorem{remark}[theorem]{Remark}
\newtheorem{Rem}[theorem]{Remark}

\numberwithin{equation}{section}

\newcommand{\A}{\mathcal A}
\newcommand{\B}{\mathcal B}
\newcommand{\al}{\alpha}
\newcommand{\be}{\beta}
\newcommand{\C}{ \mathbb C}
\newcommand{\K}{ \mathbb K} 
\newcommand{\N}{\mathbb N}

\newcommand{\DB}{{\mathcal D}_{\bar{\mathcal B}}}

\newcommand{\FF}{{\mathcal F}}
\newcommand{\g}{{\gamma}}

\newcommand{\no}{\text{:}}

\newcommand{\affb}{\mathfrak b}
\newcommand{\affg}{\mathfrak g}
\newcommand{\affh}{\mathfrak h}
\newcommand{\affn}{\mathfrak n}

\newcommand{\heis}{\mathfrak{hs}}

\newcommand{\finb}{\bar{\mathfrak b}}
\newcommand{\fing}{\bar{\mathfrak g}}
\newcommand{\finh}{\bar{\mathfrak h}}
\newcommand{\finn}{\bar{\mathfrak n}}

\newcommand{\h}{\mathfrak{h}}
\newcommand{\+}{\mathop{\oplus}}

\newcommand{\der}{\partial}

\newcommand{\la}{\lambda}
\newcommand{\mf}{\mathfrak}
\newcommand{\Mg}{M}
\newcommand{\Mgo}{M_{0}}
\newcommand{\ov}{\overline}
\newcommand{\Res}{ {\rm Res} }

\newcommand{\Vg}{V^{\kappa}(\affg)}
\newcommand{\Vgc}{V^{\kappa}(\affg)}
\newcommand{\Vgo}{V^{\kappa}_0(\affg)}
\newcommand{\Vgco}{V^{\kappa}_0(\affg)}
\newcommand{\vac}{|0\rangle}
\newcommand{\ww}{\texttt{w}}
\newcommand{\WZ}{W_\Z(\la)}

\newcommand{\Z}{ \mathbb Z }
\newcommand{\zz}{ \mathfrak z }
\newcommand{\ZZ}{ \mathcal Z }

\newcommand{\ra}{\rightarrow}

\def\l{{\langle}}
\def\r{{\rangle}}

 \DeclareMathOperator{\ad}{ad}

{\vskip-\lastskip\medskip
 \noindent
 {\em #1.}\enspace
 }%
{\qed\par\medskip
 }

\begin{document}
\title[Modular affine vertex algebras and baby Wakimoto modules]
{Modular affine vertex algebras and baby Wakimoto modules}

\author[T. Arakawa]{Tomoyuki Arakawa}
\address{
RIMS, Kyoto University, Kyoto
606-8502, Japan} \email{arakawa@kurims.kyoto-u.ac.jp}

\author[W. Wang]{Weiqiang Wang}
\address{Department of Mathematics, University of Virginia,
Charlottesville, VA 22904} \email{ww9c@virginia.edu}

\begin{abstract}
We develop some basic properties such as $p$-centers 
of affine vertex algebras and  free field vertex algebras in prime characteristic.
We show that the Wakimoto-Feigin-Frenkel homomorphism preserves the $p$-centers
by providing explicit formulas.
This allows us to formulate the notion of baby Wakimoto modules, which in particular
provides an interpretation in the
context of modular vertex algebras for Mathieu's
irreducible character formula of modular affine Lie algebras at the critical level. 
\end{abstract}

\maketitle

\date{}

\section{Introduction}

Let $\K$ be an algebraically closed field of prime characteristic $p$. 
Denote by $U=\K \otimes_\Z U_\Z$, where $U_\Z$
is the Kostant-Garland $\Z$-form (including divided powers) of the universal enveloping algebra of $\affg$. 
Mathieu \cite{Ma} established a character formula for the irreducible highest weight $U$-module $L(-\rho)$ at the critical level (see \eqref{eq:rho}), 
which can be rephrased as that
the Wakimoto module of highest weight $-\rho$ over the complex field $\C$ remains irreducible over $U$ after reduction modulo $p$. 
Mathieu also gave a character formula for $\mf l(-\rho)$ (and also for $L((p-1)\rho)$); see \eqref{eq:char}-\eqref{eq:p-1rho}.
Here $\mf l (-\rho)$ denotes the irreducible quotient $\affg$-module  of the Verma $\affg$-module of high weight $-\rho$,
which can be regarded as an irreducible module over the restricted enveloping algebra ${\mf u}_0(\affg)$ (and ${\mf u}_0(\affg) \subset U$).
These two irreducible character formulas are equivalent by the Steinberg tensor product theorem and noting that $(p-1)\rho$ is a restricted weight. 

Modular vertex algebras (i.e., vertex algebras in prime characteristic) were first considered in \cite{BR}
by Borcherds and Ryba in their study of modular moonshine. 
This paper is motivated by putting Mathieu's result in a proper context of modular Lie algebras and modular vertex algebras
(where the algebra $U$ plays no role). 
We formulate the notion of $p$-centers for vertex algebras associated to Heisenberg algebras, affine algebras,
and some other free fields, and this gives rise to corresponding $p$-restricted vertex algebras. 
We show that the $p$-centers and the state-field correspondence for these vertex algebras are
compatible in a simple manner; cf. Proposition~\ref{prop:commYi}. 

Wakimoto modules (over $\C$) were introduced by Wakimoto \cite{Wak} for $\mathfrak{sl}_2$
and then by Feigin and E.~Frenkel for general semisimple Lie algebras \cite{FF}. 
Wakimoto modules have played a fundamental role in the affine vertex algebra setting and applications
to the geometric Langlands program, cf. \cite{Fr1, Fr2}. The construction of Wakimoto modules
relies on the Wakimoto-Feigin-Frenkel homomorphism $\ww$ from an affine vertex algebra to a bosonic free field vertex algebra. 
As a main result of this note  we show that $\ww$ (over the field $\K$) preserves the $p$-centers, and indeed we provide explicit
formulas for the restriction of $\ww$ on the $p$-center.
This allows us to formulate a notion of baby Wakimoto modules, which is analogous to the more
familiar notion of baby Verma modules for modular Lie algebras. 
Now Mathieu's result can be restated that the baby Wakimoto module of highest weight $-\rho$
is irreducible as  module over $\affg$ or over ${\mf u}_0(\affg)$ (that is, it coincides with $\mf l(-\rho)$ in the above notation). 

This paper is organized as follows. 
In Section~\ref{sec:VA}, we prove some basic properties of the modular affine vertex algebras including the $p$-centers.
In Section~\ref{sec:Wakimoto}, we describe the $p$-centers of the Heisenberg vertex algebra and of a symplectic bosonic vertex algebra.
We formulate the main construction of the baby Wakimoto modules.
In Section~\ref{sec:proof}, we establish the formulas for the WFF homomorphism on the $p$-center of the affine vertex algebra.
In Section~\ref{sec:irred}, we give a reformulation of Mathieu's main result in terms of the irreducibility of the baby Wakimoto module
of highest weight $-\rho$.  
We end with some conjectures and open problems on further development of modular representation theory of affine Lie algebras. 

\vspace{.2cm}
{\bf Acknowledgments.}
We have been working on this project on and off since 2007. 
The results were presented in the 
Taitung Workshop on ``Group theory, VOA and algebraic combinatorics", Taiwan, in March 2013, organized by C.-H. Lam.
There is some overlap of our work with a  recent paper by Li and Mu \cite{LM}, where one can find more
references on modular vertex algebras in recent years. 
The first author is partially supported by JSPS KAKENHI Grant Numbers
25287004, 26610006.
The second author is partially supported by an NSF grant DMS-1405131.

\section{Modular affine algebras and modular vertex algebras}
\label{sec:VA}

\subsection{Affine Lie algebra in prime characteristic}
\label{sec:affinep}

Let $\fing$ be a finite-dimensional semisimple Lie algebra, which
is a Lie algebra of a simply connected algebraic group $\bar{G}$ over an
algebraically closed field $\K$ of characteristic $p>0$.  
Then $\fing$ is a restricted Lie algebra (also called a $p$-Lie algebra)
with $p$-power map denoted by
$-^{[p]}$; cf. \cite{Jan} for a review of modular Lie algebras. 
Moreover, $\fing$ affords a non-degenerate bilinear
form $\l \cdot, \cdot \r$, which induces a linear isomorphism $\fing \rightarrow
\fing^*$. We fix a Chevalley basis $h_i (1\le i \le \ell),
e_\alpha, f_\al (\al \in \bar{\Delta}^+)$ of $\fing$, where
$\bar{\Delta}^+$ is a set of positive roots for $\fing$
corresponding to a set of simple roots $\bar{\Pi} =\{\al_1,\dots,
\al_\ell\}.$ We further write $e_i =e_{\al_i}, f_i =f_{\al_i}$.
We denote by $\bar{B}$ (respectively, $\bar{B}_-$)
the Borel subgroup of $\bar{G}$ whose Lie algebra $\finb$
(respectively, ${\finb}_-$) is spanned by root vectors from
$\bar{\Delta}^+$ (respectively, $\bar{\Delta}^-
=-\bar{\Delta}^+$).

We consider the affine Lie algebra
$$\affg \cong L\fing \oplus \K c 
$$
where $L\fing \cong \K [t,t^{-1}] \otimes \fing$. We shall write
$x_n =t^n \otimes x$ for $x\in \fing$ and $n\in\Z$. Then $\fing$
is naturally a Lie subalgebra of $\affg$ by the identification $1
\otimes \fing \cong \fing$. We denote by $h^\vee$ the dual Coxeter
number for the affine Lie algebra $\affg$.

A Cartan subalgebra $\affh$ and Borel subalgebra $\affb$ of
$\affg$ is
$$\affh =\finh +\K c 
$$
and a Borel subalgebra is $\affb =\K c  + t \K[t] \otimes \fing +
\finb$ with nilradical $\affn =t \K[t] \otimes \fing + \finn$, so
that $\affg =\mathfrak n^- \oplus \mathfrak h \oplus \mathfrak n$.
Denote by $\Delta_+$ the set of positive roots associated to $\mf n$,
and by $\Delta_+^{\rm re}$ the subset of real roots in $\Delta_+$. 
Let $\affg^*$ denote the restricted dual of $\affg$ associated to
the root space decomposition of $\affg$.

Denote by $\bar{T} \subset \bar G$ the maximal torus with Lie algebra $\finh$.
Let $\K^* =\K -\{0\}$ be the torus corresponding to the derivation $d$ on $\affg$, where $[d,c]=0$
and $[d, t^n \otimes x] =-n t^n \otimes x$ for $x\in \fing$ and $n\in \Z$. 
Set $T =\bar{T} \times \K^*$.

\begin{lemma} [cf. \cite{Ma}, (1.4)]
There is a  restricted Lie algebra structure on the
affine Lie algebra $\affg$ as an extension of the one on $\fing$,
whose $p$-power map is given by
$$c^{[p]} =c,\quad 
(t^n\otimes x)^{[p]} = t^{np} \otimes x^{[p]}, \quad  \text{ for }
n\in\Z, x\in \fing.
$$
\end{lemma}

Then as usual one has the $p$-center $\ZZ_0(\affg)$ in the
enveloping algebra $U(\affg)$ which is generated by $x^p -x^{[p]}$
for all $x\in \affg$. The subalgebra 
of $\ZZ_0(\affg)$ generated by $x^p -x^{[p]}$
for all $x\in L\fing$ will be denoted by $\ZZ_0'(\affg)$ and referred to as
the {\em proper $p$-center}. 

Each $\chi \in (L \fing)^*$ defines a
$p$-character and gives rise to the reduced enveloping algebra by
$${\mf u}_\chi(\affg) = U(\affg) /I_\chi$$
where $I_\chi$ is the ideal generated by $a^p -a^{[p]} -\chi(a)^p$
for all $a\in L \fing$. In particular, ${\mf u}_0(\affg)$ is called the
restricted enveloping algebra of $\affg$. Note that according to our definition
$c^p-c$ is not in the ideal $I_0$.

A distinguished restricted Lie subalgebra of $\affg$ is the
Heisenberg algebra
$$
\heis = L\finh \oplus \K c  =\heis^- \+\affh \+\heis^+,
$$
where $\heis^\pm =\oplus_{n\in \pm\N} t^n \otimes \finh$. The Lie
algebra $\heis$ has a large center spanned by $c, t^{pn} \otimes
\finh$ for $n\in \Z$.

The $p$-center $\ZZ_0 (\heis)$ of $U(\heis)$ is generated by $x^p
-x^{[p]}$ for all $x\in \heis$ and the proper $p$-center $\ZZ_0' (\heis)$ of $U(\heis)$ is
by definition the subalgebra generated by $x^p
-x^{[p]}$ for all $x\in L\finh$. The whole center of $U(\heis)$ is
generated by $\ZZ_0 (\heis)$ and $c, t^{pn} \otimes
\finh$ for $n\in \Z$, though  this fact will not be needed below.

\subsection{Vertex algebras in prime characteristic}
The usual notion of vertex algebras can be readily made sense over
the field $\K$ of characteristic $p>0$ (cf. Borcherds-Ryba
\cite{BR}). All one needs is to use the divided power of the
translation operator $T^{(i)} =T^i/i!, i\ge 1$ and noting that
$$
Y(T^{(i)} a, z) = \der^{(i)} Y(a, z),
$$
where $\der^{(i)}$ denotes the $i$th divided power of the derivative
with respect to $z$.

Denote $L_+\fing =\sum_{n \in\Z_+} t^n \otimes \fing$. It is well known that the vacuum
$\affg$-module of level $\kappa \in \K$
$$
\Vg = U(\affg) \bigotimes_{U(L_+\fing  +\K c)} \K_\kappa
$$
carries a canonical structure of a vertex algebra (cf. e.g.
\cite{Fr2}), where $L\fing^+$ acts on $\K_\kappa =\K$ trivially and $c$ as
scalar $\kappa$. Denote by $\vac =1\otimes 1$ the vacuum vector
in $\Vg$.

\subsection{The $p$-centers of modular vertex algebras}

Let
$$
x(z) =\sum_{n\in\Z} x_n z^{-n-1}, \qquad x \in \fing.
$$
The following lemma on vertex operators is standard (cf.
\cite{Fr2}), except the divided power notation.

\begin{lemma} \label{lem:state-field}
The following formulas hold in the vertex algebra $\Vg$:  
\begin{eqnarray}
Y(x_{-r} \vac, z)= \partial^{(r-1)} x(z)  &=& \sum_{n\in\Z}
\binom{-n-1}{r-1}x_n z^{-n-r},  \label{statefield} \\
Y(x_{-r_1} y_{-r_2}\cdots  \vac, z) &=& \no
\partial^{(r_1-1)} x (z)\;\partial^{(r_2-1)} y (z) \cdots\no  \nonumber
\end{eqnarray}
for $x,y,
\ldots \in \fing$, and  $r, r_1, r_2, \ldots \in \N.$
\end{lemma}

\begin{lemma} \label{lem:pVO}
The following identities hold for the vertex algebra $\Vg$:
for $x\in \fing$ and $r\ge 1$, we have
\begin{eqnarray}
Y ( x_{-rp}\vac, z ) = \der^{(rp-1)} x(z)& =& \sum_{n\in \Z}
\binom{-n-1}{r-1} x_{np} \, z^{-np-rp},  \label{Ypower}
 \\
Y (x_{-r}^p \vac, z ) = \no (\der^{(r-1)} x(z))^p \no &=&
\sum_{n\in\Z} \binom{-n-1}{r-1} x_n^p \, z^{-np-rp}. \label{Yp}
 \end{eqnarray}
\end{lemma}
The special case of Lemma~\ref{lem:pVO} for $r=1$ reads:
\begin{eqnarray}
Y (x_{-p}\vac, z ) = \der^{(p-1)} x(z)& =& \sum_{n\in \Z} x_{np}\,
z^{-np-p},
 \\
Y ( x_{-1}^p  \vac, z ) = \no x(z)^p \no &=& \sum_{n\in\Z} x_n^p
\, z^{-np-p} .
 \end{eqnarray}

To prove Lemma~\ref{lem:pVO}, we shall need the following classical formula.
\begin{lemma} \label{binom}
For $a = a_0 +pa' \in \Z_{\ge 0}, b = b_0 +pb'$ with $0\le a_0, b_0 \le p-1$ and
$a' \ge 0$, we have
$$\binom{b}{a} \equiv  \binom{b'}{a'} \binom{b_0}{a_0} \mod p.
$$
%
(All the $a$'s and $b$'s involved are integers.)
\end{lemma}

\begin{proof}[Proof of Lemma~\ref{lem:pVO}]
By Lemma~\ref{binom}, we obtain that $\binom{-m-1}{rp-1} \equiv 0 \mod p$ if $p
\nmid m$, and $\binom{-np-1}{rp-1} \equiv \binom{-n-1}{r-1} \mod p$  for
$n\in\Z$. Now \eqref{Ypower} follows from \eqref{statefield}.

We write $A(z) \equiv  \der^{(r-1)} x(z) =\sum_{n\in\Z}
\binom{-n-1}{r-1} x_n z^{-n-r} =A_+(z) + A_-(z)$, where $A_\pm (z)
=\sum_{n \lesseqgtr -r}\binom{-n-1}{r-1} x_n z^{-n-r}$. By the
definition of normal ordered product and induction on $m \geq 1$,
we have
\begin{eqnarray*}
\no A(z)^m \no &=& A_+(z) \no A(z)^{m-1} \no + \no A(z)^{m-1} \no
A_-(z) \\
&=& \sum_{i=0}^m \binom{m}{i} A_+(z)^i A_-(z)^{m-i}.
\end{eqnarray*}
Note that $A_+(z)^p = \sum_{n \le -r} \binom{-n-1}{r-1} x_n^p
z^{-np-p}$ since $x_n$ with $n<0$ commute and $b^p =b$ for $b\in
\mathbb F_p$. Similarly, $A_-(z)^p =\sum_{n\ge 0} \binom{-n-1}{r-1} x_n^p
z^{-np-p}$. Hence, $ \no A(z)^p \no = A_+(z)^p + A_-(z)^p$,
whence \eqref{Yp}.
\end{proof}

\begin{remark}
Lemmas~\ref{lem:state-field} and \ref{lem:pVO} are applicable to
other modular vertex algebras, e.g. $\FF$.
\end{remark}

Denote
$$
\iota (x_n) = x_n^p -x_{np}^{[p]}, \quad x\in \fing.
$$
We also denote $\iota (z) =z^p$ (the Frobenius morphism). In the next
proposition, which follows directly from Lemmas~\ref{lem:state-field}
and \ref{lem:pVO}, we formulate a basic property of affine vertex algebras.

\begin{proposition} [Commutativity of $\iota$ and $Y$]
 \label{prop:commYi}
For $x\in \fing$ and $r\ge 1$, we have
 \begin{eqnarray}  \label{p-cent-field}
Y( \iota (x_{-r} )\vac, z)
 &=& Y \left( (x_{-r}^p -(x^{[p]})_{-rp})\vac, z \right) \nonumber   \\
 &=& \left (\partial^{(r-1)} x(z) \right )^p
 -  \partial^{(rp-1)} x^{[p]}(z)= \iota \left(\partial^{(r-1)} x(z) \right).
\end{eqnarray}
\end{proposition}
When $r=1$, we have
\begin{eqnarray*} \label{p-cent-field1} Y( \iota
(x_{-1} )\vac, z) = \iota Y(x_{-1} \vac, z) = \sum_{n\in\Z} (x_n^p
-(x^{[p]})_{np} )\, z^{-np-p}.
\end{eqnarray*}

By definition, the center of a vertex algebra $V$ consists of all vectors $v\in V$ such that $Y(a,z)v \in V[[z]]$ for all $a\in V$. 
The center of a vertex algebra is a commutative vertex algebra (cf. \cite{Fr2}).

\begin{definition}
The {\em $p$-center} (or the {\em Frobenius center}) $\zz_0 (\Vg)$ of the vertex algebra $\Vg$ is
defined to be the subspace $\ZZ_0'(\affg) \vac \subset \Vg$.
\end{definition}
Clearly these $p$-centers (and other $p$-centers below) are vertex subalgebras of the centers of the
corresponding vertex algebras. 

\begin{proposition} \label{prop:resVA}
\begin{enumerate}
\item The $p$-center $\zz_0 (\Vg)$ is a commutative vertex
subalgebra of $\Vg$.

\item  $U(\affg) \cdot \zz_0(\Vg)$ is an ideal of the
vertex algebra $\Vg$, and so the quotient 
$$
\Vgo
\stackrel{\text{def}}{=} \Vg /(U(\affg) \cdot \zz_0(\Vg))
$$ 
carries
an induced vertex algebra structure.  
\end{enumerate}
\end{proposition}

\begin{proof}
Part (2) follows from (1) easily, and we will prove (1).

Observe that the Fourier components of \eqref{p-cent-field} are of
the form $x_n^p -(x^{[p]})_{np}$ (up to a scalar multiple), and
hence belong to the $p$-center $\zz_0 (\Vg)$. By definition, the
$p$-center $\zz_0 (\Vg)$ is spanned by elements of the form $x =
\iota(a_{-r_1} b_{-r_2} \cdots \vac)$ with $a,b\ldots \in \fing$
and $r_1,r_2,\ldots >0$. By Proposition~\ref{prop:commYi}, the vertex operator
$$Y(x,z) = \no Y(\iota(a_{-r_1}) \vac,z) Y(\iota(b_{-r_2})\vac,z)
\cdots\no$$ is a linear combination of operators composed from
those of the form $x_n^p -(x^{[p]})_{np}$, and hence clearly
preserves $\zz_0 (\Vg)$.
\end{proof}

Following the standard terminology in the theory of modular Lie algebras,
we shall refer to the vertex algebras $\Vgo$ as the
{\em restricted} (or more precisely {\em $p$-restricted}) vertex
algebras associated to $\affg$.

\begin{remark}
Since $c^p-c \not \in \ZZ_0'(\affg)$ by definition, the central charges for the restricted vertex
algebras $\Vgo$  can be any scalar in $\K$. 
\end{remark}

A {\em baby Verma $\affg$-module} (associated to a weight $\la$ on $\affh$ of level $\kappa$)
is a $\affg$-module of the form
$$
V(\la) \equiv V^\kappa(\la) = \mf u_0(\affg) \bigotimes_{\mf u_0(\affn +\affh)}
\K_{{\la}}
$$
where $\affn$ acts trivially on the one-dimensional space
$\K_{{\la}} \cong \K$  and $\affh$ acts by the weight ${\la} \in \affh^*$. 
These baby Verma modules are modules of the restricted vertex algebra $\Vgo$. 

\section{The baby Wakimoto modules}
 \label{sec:Wakimoto}
\subsection{A vertex algebra $\Mg$}

Let $\mathcal A^\affg$ be the Weyl algebra over $\K$ with
generators $a_{\al,n}, a^*_{\al,n}$ with $\al \in \ov{\Delta}_+,
n\in \Z$, and relations
$$
[a_{\al,n}, a^*_{\beta,m}] =\delta_{\al,\beta} \delta_{n,-m},
 \quad [a_{\al,n}, a_{\beta,m}] =
 [a^*_{\al,n}, a^*_{\beta,m}] =0.
$$
A restricted Lie algebra structure on $\mathcal A^\affg$ is given
as follows:
$$
a_{\alpha,n}^{[p]} = (a^*_{\al,n})^{[p]} =0,\quad n\in \Z.
$$

Introduce the fields
$$a_\al (z) =\sum_{n\in\Z}a_{\al,n} z^{-n-1}, \quad
 a^*_\al (z) =\sum_{n\in\Z}a^*_{\al,n} z^{-n}, \qquad\al \in \ov{\Delta}_+.
$$

Let $\Mg$ be the Fock representation of $\mathcal A^\affg$
generated by $\vac$ such that
$$a_{\al,n} \vac =0, \quad n\ge 0; \qquad
 a^*_{\al,n} \vac =0, \quad n > 0.
$$
As a vector space, $\Mg \cong \K [a_{\al,n-1}, a^*_{\al,n}]_{\al
\in \bar \Delta_+, n\le 0}$. It is well known that $\Mg$ carries a
vertex algebra structure with state-field correspondence
\begin{align*}
Y& (a_{\al_1,-r_1}\cdots a_{\al_k,-r_k} a^*_{\beta_1,-s_1} \cdots
a^*_{\beta_m,-s_m} \vac) \\
 &= \no \der^{(r_1-1)} a_{\al_1}(z) \cdots  \der^{(r_k-1)} a_{\al_k}(z)
  \der^{(s_1)} a^*_{\beta_1}(z) \cdots \der^{(s_m)}
  a^*_{\beta_m}(z)\no
\end{align*}
and with the translation operator $T$ such that
$$T\vac =0, \; [T, a_{\al,n}] =-na_{\al,n-1},\,
 [T, a^*_{\al,n}] =-(n-1) a^*_{\al,n-1}.
$$

\begin{proposition} \label{pcenter=center}
\begin{enumerate}
\item The $p$-center $\mathcal Z_0 (\mathcal A^\affg)$ is equal to
$\K [a_{\al,n}^p, (a^*_{\al,n})^p]_{\al \in \bar \Delta_+, n\in
\Z}$; and moreover, $\mathfrak z_0 (\Mg) \cong \K [a_{\al,n-1}^p,
(a^*_{\al,n})^p]_{\al \in \bar \Delta_+, n\le 0}$.

\item The space $U(\mathcal A^\affg) \cdot \zz_0(\Mg)$ is an ideal
of the vertex algebra $\Mg$, so the quotient 
$$
\Mgo := \Mg
/(U(\mathcal A^\affg) \cdot \zz_0(\Mg))
$$ 
carries an induced vertex
algebra structure.
\end{enumerate}
\end{proposition}

\subsection{Realization of the contragredient Verma modules}

We first recall (cf. e.g. \cite[pp.135]{Fr2}) that the
contragredient Verma module of $\fing$ can be realized via its
identification with the space of regular functions $\mathcal
O_{\bar N_+}$ on the unipotent subgroup $\bar N_+$ of the algebraic
group $\bar{G}$; equivalently, this is described as follows:
let $\bar U=\bar{N}_+ \bar{B}_-/\bar{B}_-$ be the open cell of the flag variety $\mathcal B :=\bar{G} /\bar{B}_-$. 
Let  
\begin{equation}  \label{eq:phi}
\phi: U(\fing) \longrightarrow\DB(\bar U)
\end{equation}
denote the
composition of the restriction to the open cell
$\Gamma(\bar{\mathcal B}, \DB) \rightarrow \DB(\bar U)$ with an
algebra homomorphism $U(\fing) \rightarrow \Gamma(\bar{\mathcal
B}, \DB)$ where $\DB$ denotes the sheaf of crystalline
differential operators on the flag variety $\bar{\mathcal B}$ (i.e. no divided powers of differential
operators, cf. e.g., \cite{BMR}). Then
the contragredient Verma module of $\fing$ is the pullback of the $\DB(\bar U)$-module
$\mathcal O_{\bar U}$ via  the algebra
homomorphism $\phi$. The restriction $\phi |_{\fing}: \fing \rightarrow \text{Vect}_{\bar U}$
is a Lie algebra homomorphism, where $\text{Vect}_{\bar U}$ the vector fields over $\bar U$. 

%

Let us fix some coordinates $y_\g$ (and $\der_\g:=\frac{\der}{\der
y_\g}$) for the open cell $\bar U$. The following lemma is
standard.
\begin{lemma} \label{lem:pcenterWeyl}
We have $\mathcal Z_0 (\DB (U)) =\K [y_\g^p, \der_\g^p]_{\g \in
\bar \Delta_+}$.
\end{lemma}

The following is known (cf., e.g.,  \cite[\S1.3]{BMR}). 

\begin{lemma}  \label{lem:pcenters}
The restriction $\phi: \fing  \longrightarrow \text{Vect}_{\bar U}$ is a homomorphism of restricted Lie algebras, and
the homomorphism $\phi: U(\fing) \longrightarrow\DB(\bar U)$ maps the
$p$-center of $U(\fing)$ to the $p$-center of $\DB(\bar U)$. 
\end{lemma}

In some cases, we can make this fairly explicit as
follows. For $x \in \finn \+ \finh$, one can write
 \begin{equation}  \label{eq:phix}
 \phi (x) = \sum_{\beta \in \bar \Delta_+} c_\beta m_\beta
(y_\g) \der_\be
\end{equation}
where $c_\beta \in \K$ and $m_\beta(y_\g)$ denotes some monomials
in the variables $y_\g$, for $\g \in \bar \Delta_+$.

\begin{lemma} \label{pFormula}
Let  $x \in \finn \+ \finh$ and retain the above notation \eqref{eq:phix}. Then we have
 $$\phi (\iota (x))
 = \sum_{\beta \in \bar \Delta_+} c_\beta^p m_\beta (y_\g^p)
 \der_\be^p.$$
\end{lemma}

\begin{proof}
We know that
$$\phi(\iota (x)) =\phi(x)^p - \phi(x^{[p]}) =
\big (\sum c_\beta m_\beta (y_\g) \der_\be \big)^p- \phi(x^{[p]})
$$
lies in the $p$-center $\mathcal Z_0 (\DB (U))$ and
$\phi(x^{[p]})$ is a sum of differential operators of order one
(plus some possible constants). We expand this $p$th power and
move the differential operator $\der_\be$ to the right by using
commutators. Lemma~\ref{lem:pcenterWeyl} ensures all the
commutators will cancel out with each other as they would produce
differential operators of order between $1$ and $p-1$.
\end{proof}

\begin{Rem} 
Let $\mu: T^* \bar {\mathcal{B}}\rightarrow  \bar {\mathcal{N}}$ be the 
Springer resolution,
$\mu^{(1)}:T^*  \bar {\mathcal{B}}^{(1)}\rightarrow  \bar
 {\mathcal{N}}^{(1)}$
be the induced map between the corresponding Frobenius twists.
Then the same argument as in proof of Lemma 
\ref{pFormula}
shows that 
$
\phi(\iota(x))=(\mu^{(1)})^*(\bar x)|_{U}
$
(\cite[1.3.3]{BMR}),
where $\bar x$ is the image of $x\in \fing$
by the projection $\K[\fing^*]\ra \K[\bar {\mathcal{N}}]$.

\end{Rem}
\subsection{Heisenberg vertex algebra $\pi^\kappa$}

Let $\B^\h_\kappa$ be a copy of Heisenberg algebra (of the affine Lie algebra $\affg$),
with generators $\bf 1$ and $b_{i,n}$\; $(i=1,\ldots, \ell, n\in \Z)$ and  subject to the relations
$$
[b_{i,n}, b_{j,m}] =n \kappa \l h_i, h_j \r
\delta_{n,-m} {\bf 1}.
$$

Denote by $\pi^\kappa$ the vertex algebra $\K [b_{i,n}]_{1\le i\le
\ell; n<0}$ (where $\bf 1$ acts as the identity map), with 
$$
Y(b_{i,-1}, z) \equiv b_i(z) = \sum_{n<0} b_{i,n} z^{-n-1},
$$
and the translation $T$ given by
$$
T \cdot b_{i_1,n_1} \cdots b_{i_m,n_m}
 = -\sum_{j=1}^m n_j  b_{i_1,n_1} \cdots b_{i_j,n_j-1}\cdots
 b_{i_m,n_m}.
$$
A restricted Lie algebra structure on $\B^\h_\kappa$ is
given by
 $$b_{i,n}^{[p]} =b_{i, np}, \quad {\bf 1}^{[p]} ={\bf 1}, \qquad 1\le i\le \ell, n \in \Z.
 $$
The proper $p$-center of the vertex algebra $\pi^\kappa$ (which excludes ${\bf 1}^p-\bf 1$)
is then equal to
$$
\zz_0' (\pi^\kappa) = \K [b_{i,n}^p - b_{i, np}]_{1\le i\le \ell; n<0}.
$$
Denote by $\pi^\kappa_0$ the quotient vertex algebra of $\pi^\kappa$ by the ideal (in the
sense of vertex algebras) generated by $\zz_0' (\pi^\kappa)$.

For $\kappa=0$, $\pi^0$ is naturally a commutative vertex algebra.

\subsection{The Wakimoto-Feigin-Frenkel (WFF) homomorphism}

Let $\kappa_c$ denote the critical level  for $\affg$.
There exists a homomorphism of vertex algebras
$$
\ww =\ww_\kappa: \Vgc \rightarrow \Mg \otimes \pi^{\kappa-\kappa_c},
$$
which is roughly speaking an affinization of
$\phi$ defined in \eqref{eq:phi}; see \cite[Theorem~6.1.6]{Fr2}.
We shall call $\ww$ the {\em WFF homomorphism}, 
since this was introduced by Wakimoto \cite{Wak} in the $\mathfrak{sl}_2$
case and by Feigin-Frenkel \cite{FF} for general semisimple Lie algebras $\fing$.  
On generating fields, the formulas for $\ww$ read as follows:
\begin{align}
e_i(z) \mapsto & a_{\al_i}(z) +\sum_{\al_i \neq \be \in \bar{\Delta}_+} \no P_\be^i \big(a_\al^*(z)\big) a_\be(z)\no, 
 \label{eq:e}
  \\
h_i(z) \mapsto & \sum_{ \be \in \bar{\Delta}_+}  \beta(h_i) \no a_\be^*(z) a_\be(z)\no +b_i(z), 
 \label{eq:h}
  \\
f_i(z) \mapsto & \sum_{\be \in \bar{\Delta}_+} \no Q_\be^i \big(a_\al^*(z)\big) a_\be(z)\no
 + \big(c_i+(\kappa-\kappa_c)\l e_i,f_i\r \big) \partial_za_{\al_i}^*(z) + a_{\al_i}^*(z) b_i(z).
  \label{eq:f}
\end{align}
We shall take for granted that $c_i$ are integers.
The polynomial $Q_\be^i\big(a_\al^*(z)\big)$ for $\beta=\al_i$  can be determined explicitly as follows; cf. \cite{Fr2}.

\begin{lemma}  \label{lem:Qi}
We have $Q_{\al_i}^i\big(a_\al^*(z)\big) = - \no a_{\al_i}^*(z)^2 \no$.
 \end{lemma}

\begin{example}
For $\fing
=\mathfrak{sl}_2$, the formulas for $\ww$ are greatly simplified (where we drop the
indices of the Chevalley generators and of the free fields) as follows:
\begin{align}  \label{sl2}
\begin{split}
e(z) \mapsto & a(z),   \qquad h(z)  \mapsto -2 \no a^*(z) a(z)\no
+b(z),
 \\
 f(z)  \mapsto & - \no a^*(z)^2 a(z)\no +\kappa \partial_z a^*(z) +a^*(z) b(z).
\end{split}
\end{align}
\end{example}


The following is a main result of this paper, 
which will be proved in Section~\ref{sec:proof}.

\begin{theorem} \label{th:pcenter}
The homomorphism $\texttt{w}: \Vgc \rightarrow \Mg \otimes \pi^{\kappa-\kappa_c}$
sends the $p$-center $\zz_0 (\Vgc)$ to the $p$-center $\zz_0 (\Mg)
\otimes \zz_0 (\pi^{\kappa-\kappa_c})$.
\end{theorem}

We have the following immediate consequence.
\begin{corollary}  \label{babyW}
The homomorphism $\ww$ induces naturally a homomorphism of vertex
algebras $\ww_0: \Vgco \rightarrow \Mgo \otimes \pi^{\kappa-\kappa_c}_0$.
\end{corollary}

\subsection{The baby Wakimoto modules}

We define baby Wakimoto modules (associated to $p$-characters)
 at the critical level $\kappa_c$ as follows. 
 
Let $\xi$ be a $p$-character on the Lie algebra $\A^\affg$ such that 
 \begin{equation} \label{eq:pch}
 \xi((a_{\alpha, n}^*)^p) =0 = \xi(a_{\alpha, n-1}^p), 
 \qquad \text{ for  } n > 0, \alpha\in \bar{\Delta}_+.
 \end{equation}
 Let $\xi^\pi$ be a $p$-character on the Lie algebra $\B^\affh$.
Take a
weight $\la (t) =\sum_{n\in\Z} \la_n t^{-n-1} \in \finh^* ((t))$ which is
compatible with the $p$-character $\xi^\pi$ in the sense that
$$
\la_{i,n}^p -\la_{i,np}=\xi^\pi (b_{i,n})^p, \quad \text{ for all } 1\le i\le \ell,n \in \Z,
$$
where $\la_{i,n} = \la_n (h_i)$. Such a weight
gives rise to the one-dimensional $\pi^0$--module
$\K_{\la(t)}$ on which $b_{i,n}$ acts by multiplication by
$\la_{i,n}$. This defines a $\affg$-module at the
critical level on $\Mg$ (which is identified with $\Mg \otimes \K_{\la(t)}$).
Identifying $\Mg$ as the polynomial algebra $\K [a_{\al,n-1}, a^*_{\al,n}]_{\al
\in \bar \Delta_+, n\le 0}$, we let
 $I_{\xi}$ be the subalgebra of $\Mg$ spanned by
$(a_{\alpha, n}^*)^p - \xi((a_{\alpha, n}^*)^p), 
a_{\alpha, n-1}^p - \xi(a_{\alpha, n-1}^p)$, 
where $n\leq 0$ and $ \alpha\in \bar{\Delta}_+.$
Since all the generators of  $I_{\xi}$  are central in $U(\affg)$,
 $I_{\xi}$ is clearly a $\affg$-submodule of $\Mg$, and this gives rise to
 a quotient $\affg$-module 
 $$
 \mf w_{\xi} (\la) := \Mg/ I_{\xi}.
 $$
We refer to the $\affg$-module $\mf w_{\xi} (\la)$ as the {\em baby Wakimoto module} of high weight $\la$ (and $p$-character $\xi$).

Now assume $\kappa \neq \kappa_c$. 
Given $\la \in \affh^*$, let $\pi^{\kappa-\kappa_c}_0(\la)$
be the Fock representation of $\A^\affg$ generated by a vector $|\la\r$ such that
$$
b_{i,n} |\la\r =0 \quad (n>0), \qquad
b_{i,0} |\la\r =\la (h_i) |\la\r,
\qquad {\bf 1} |\la \r =|\la \r.
$$
Any module over the vertex algebra $\Mg \otimes \pi^{\kappa-\kappa_c}$ becomes
a module over $\Vg$ via the pullback of homomorphism $\ww: \Vg \rightarrow \Mg \otimes \pi^{\kappa-\kappa_c}$.
In particular, 
$$
W(\la) := \Mg \otimes \pi^{\kappa-\kappa_c}(\la)
$$
becomes a module over $\affg$ of level $\kappa$. 
This is the {\em Wakimoto module} of high weight $\la$ and level $\kappa$ defined in \cite{FF}. 

Let $\xi$ be a $p$-character on the Lie algebra $\A^\affg$ satisfying \eqref{eq:pch}.
 Let $\xi^\pi$ be a $p$-character on the Lie algebra $\B^\affh$
such that  
$$
\xi^\pi(b_{i, n}^p -b_{i,np})=0, \qquad \text{ for } n\geq 0. 
$$
Assume that $\la \in \affh^*$ is
compatible with the $p$-character $\xi^\pi$ in the sense that
$$
\la(h_{i})^p -\la (h_{i})=\xi^\pi (h_{i})^p, \quad \text{ for all }1\le i\le \ell.
$$
Identifying $\pi^{\kappa-\kappa_c}(\la)$ as the polynomial algebra $\K [b_{i,n}]_{n< 0, 1\le i\le \ell}$, we let
 $I_{\xi^\pi}$ be the subalgebra of $\pi^{\kappa-\kappa_c}(\la)$ spanned by
$b_{i,n}^p -b_{i,np}-\xi^\pi (b_{i,n})^p$,
where $n< 0$ and $1\le i \le \ell.$
Since all the generators of  $I_{\xi^\pi}$  are central in $U(\affg)$,
 $I_{\xi^\pi}$ is clearly a $\affg$-submodule of $\pi^{\kappa-\kappa_c}(\la)$, and this gives rise to
 a quotient $\affg$-module 
 $$
 \mf w^\kappa_{\xi,\xi^\pi} (\la) :=( \Mg \otimes \pi^{\kappa-\kappa_c}(\la)) / (I_{\xi} \otimes I_{\xi^\pi})
  \cong  (\Mg/ I_{\xi}) \otimes ( \pi^{\kappa-\kappa_c}(\la)/ I_{\xi^\pi}).
 $$
We refer to the $\affg$-module $\mf w^\kappa_{\xi, \xi^\pi} (\la)$ as the {\em baby Wakimoto module} of high weight $\la$ (and $p$-character $(\xi,\xi^\pi)$).
  
Assume for now that $\pi^{\kappa-\kappa_c}(\la)/ I_{\xi^\pi}$ is a module over the restricted vertex algebra
$\pi^{\kappa-\kappa_c}_0$. Since $\iota (b_{i,-1}) \vac=0 \in \pi^{\kappa-\kappa_c}_0$, we have
$$
0 =Y(\iota (b_{i,-1}) \vac, z) = \sum_{n\in \Z} (b_{i,n}^p- b_{i, np}) z^{-np-p}
$$ 
when acting on $\pi^{\kappa-\kappa_c}(\la)/ I_{\xi^\pi}$.  Hence $\xi^\pi =0$,
and $\la(h_i) \in \mathbb F_p$ for $1\le i\le \ell$. The converse is also true: if $\xi^\pi =0$
and $\la(h_i) \in \mathbb F_p$ for  $1\le i\le \ell$, then $\pi^{\kappa-\kappa_c}(\la)/ I_{\xi^\pi}$ is a module over the restricted vertex algebra
$\pi^{\kappa-\kappa_c}_0$. 
Similarly,  $\Mg/ I_{\xi}$ is a module over the restricted vertex algebra $\Mgo$
if and only if $\xi=0$. 
Summarizing, we have proved the following.

\begin{proposition} 
Let $\kappa \neq \kappa_c$. If $\la\in \affh^*$ satisfies $\la(h_i)\in \mathbb F_p$ for each $1\le i\le \ell$ and $\la(c)=\kappa$,
 then $\mf w^\kappa_{0,0} (\la)$
is a module over the vertex algebra $\Mgo \otimes \pi^{\kappa-\kappa_c}_0$, and hence a module over the vertex algebra $\Vgc$. 
\end{proposition}

\section{Proof of Theorem~\ref{th:pcenter}}
\label{sec:proof}

\subsection{Restriction of $\ww$ to the $p$-center}

Theorem~\ref{th:pcenter} follows readily from 
the explicit description of the restriction map of $\texttt{w}$
to $\zz_0(\Vgc)$ given in the following theorem. Recall that
\begin{equation}  \label{eq:iotaehf}
\iota(e_i  (z)) =e_i(z)^p, \qquad 
\iota(h_i  (z)) =h_i(z)^p -h_i(z), \qquad
\iota(f_i  (z)) =f_i(z)^p. 
\end{equation}

\begin{theorem}   \label{th:ehfww}
The restriction of $\texttt{w}$
to $\zz_0(\Vgc)$  is given in terms of fields as follows:
\begin{align}
\iota(e_i  (z))^\ww = & a_{\al_i}(z)^p +\sum_{\al_i \neq \be \in \bar{\Delta}_+} P_\be^i \big(a_\al^*(z)\big)^p a_\be(z)^p, 
 \label{eq:ep}
  \\
\iota(h_i  (z))^\ww = & - \sum_{ \be \in \bar{\Delta}_+}  \beta(h_i) a_\be^*(z)^p a_\be(z)^p + b_i(z)^p
-\der^{(p-1)} b_i(z), 
 \label{eq:hp}
  \\
\iota(f_i (z))^\ww = & \sum_{\be \in \bar{\Delta}_+} Q_\be^i \big(a_\al^*(z)\big)^p a_\be(z)^p
 + 
 (\kappa^p-\kappa)\l e_i,f_i \r   \big( \partial_z a_{\al_i}^*(z) \big)^p
 \notag  \\
  &\qquad \qquad +  a_{\al_i}^*(z)^p \big(b_i(z)^p -\der^{(p-1)} b_i(z) \big).
  \label{eq:fp}
\end{align}  
\end{theorem}
Actually $P_\be^i \big(a_\al^*(z)\big)^p$ and $Q_\be^i \big(a_\al^*(z)\big)^p$ 
are simply  polynomials in the commuting vertex operators $a_\al^*(z)^p$, for $\al \in \bar{\Delta}_+$. 
Note that the normal orderings are no longer needed in the above formulas.

\begin{example}
The formulas above read in the case of
$\widehat{\mathfrak{sl}}_2$ as follows:
\begin{align*}
\iota ( e(z)  )^\ww &= a(z)^p,   
 \\
 \iota ( h(z)  )^\ww &=
-2  a^*(z)^p a(z)^p +b(z)^p-\der^{(p-1)} b(z),
 \\
\iota ( f(z)  )^\ww  & = -  a^*(z)^{2p} a(z)^p +(\kappa^p-\kappa)  \partial_z a^*(z) 
+   a^*(z)^p \left (b(z)^p-\der^{(p-1)} b(z) \right).
\end{align*}
 \end{example}
 
The remainder of this section is devoted to the proof of Theorem~\ref{th:ehfww}.
 
\subsection{Proof of Theorem~\ref{th:ehfww}}

Let 
$$
A=\K[a_{\alpha,0}^*]_{\alpha\in \bar{\Delta}_+},
$$
which can be
identified with $\K[\bar N_+]=\K[\bar U]$.  The space $V :=\sum_{\alpha\in
\bar{\Delta}_+}A a_{\alpha,-1}$ is identified with
$\mathcal{T}_{\mathcal B}(\bar U)$, where $\mathcal{T}_{\mathcal
B}$ is the tangent sheaf of $\mathcal B$. Set $\Omega := TA$, where
$T$ is the translation operator in the vertex algebra $\Mg$. Note
that $(\Mgo)_0=A$ and $(\Mgo)_1=V \+ \Omega$.
An element $D\in V$ acts on $A$ as a derivation by the
correspondence $D\mapsto D_{(0)}$, where we denote $a_{(n)}
=Y(a,z)_{(n)} =\Res_{z}z^n Y(a,z)$. One has
\begin{align}  \label{commutator}
[D_{(m)}, f_{(n)}]=(D f)_{(m+n)}\quad \text{for $D\in V$ and $f\in
A$}.
\end{align}
Recall $\phi:  U(\fing)  \rightarrow \DB(\bar U)$, $x \mapsto D^x$.
The algebra $\DB(\bar U)$ can be identified with the $A$-algebra 
generated by $V$ such that $D\cdot f =f\cdot D +D_{(0)}f$, and thus we have
$D^x \in V$ for each $x\in \fing$. It follows by \eqref{eq:e}-\eqref{eq:f} that the
image of $x_{-1} \vac$ under the vertex algebra homomorphism $\ww:
\Vgc \rightarrow \Mg\otimes \pi^{\kappa-\kappa_c}$, denoted by $x_{-1}^\ww \vac$, is of the
form
\begin{align}  \label{eq:twoterm}
 x_{-1}^\ww \vac   =D^x_{(-1)} + f^x_{(-2)} +b^x_{(-1)}
\end{align}
for $x\in \fing$, some $f^x \in A$ and some conformal weight one vector $b^x \in \pi^{\kappa-\kappa_c}$ depending on $x$.

\begin{lemma} \label{lem:comm a*}
Let $x\in \fing$ and
$m\in\Z$. Then $\big(\iota(x_{-1})^\ww \vac \big)_{(m)}$  commutes with $a_{\alpha,n}^*$ for all $\al \in \bar
\Delta_+$ and $n\in \Z$.
\end{lemma}

\begin{proof}
Since the algebra homomorphism $\phi: U(\fing) \rightarrow\DB(U)$
is compatible with the restricted Lie algebra structures (see Lemma~\ref{lem:pcenters}), we have, for $f\in A$, that
\begin{align}  \label{commutator2}
&[D^x,D^y]f=D^{[x,y]}f,\quad (D^x)^p f=D^{x^{[p]}}f.
\end{align}
%
Recall from Proposition~\ref{prop:commYi}  that
\begin{align*}
Y( \iota(x_{-1})^\ww\vac,z)=\sum_{m\in \Z} \left ( (x^\ww_{m})^p -
(x^{[p]})_{mp}^\ww \right ) z^{-mp-p}.
\end{align*}
By \eqref{commutator}, \eqref{eq:twoterm} and \eqref{commutator2},
one has, for $f\in A$, 
\begin{align*}
 [(x_{m}^\ww )^p, f_{(n)}] &= (\ad x_{m}^\ww )^p (f_{(n)})
  =(\ad
D^x_{(m)})^p (f_{(n)}) \\
&=((D^x)^p f)_{(n+mp)} 
 =(D^{x^{[p]}}f)_{(n+ mp)}
\\ &=[D^{x^{[p]}}_{(mp)}, f_{(n)}] 
 =[(x^{[p]})_{mp}^\ww, f_{(n)}].
\end{align*}
Now the lemma follows by taking $f =a^*_{\al,0}$.
\end{proof}

\begin{lemma}  \label{lem:comm a}
Let $x\in \fing$ and 
$m\in\Z$. Then  $(\iota(x_{-1})^\ww \vac)_{(m)}$ commutes with  $a_{\alpha,n}$ for all $\al \in \bar
\Delta_+$ and $n\in \Z$.
\end{lemma}

\begin{proof}
Recall from \cite{Fr1} that, for $\alpha\in \bar{\Delta}_+$,
$e_{\al,-1}^\ww$ is of the form
\begin{align}  \label{eq:w}
e_{\al,-1}^\ww 
=a_{\alpha,-1}
+\sum_{\beta\in \bar{\Delta}_+\atop \beta> \alpha}
P_{\beta}^{\alpha}a_{\beta,-1},
\end{align}
where 
$P_{\beta}^{\alpha}$ is a polynomial in $A=\K[a_{\alpha,0}^*]_{\alpha\in \bar{\Delta}_+}$ of weight $\al-\beta$.
Indeed, the weight of the right-hand-side of \eqref{eq:w} equals
$\alpha$ (see for example the first line of \cite[\S 1.3]{Fr1}),
and hence the coefficient $P_{\beta}^{\alpha}\in A$ must be zero
unless $\beta>\alpha$ in the standard dominance order of the root
lattice of $\fing$. It follows that each $a_{\alpha,-1}$ can be
written as
\begin{align}
a_{\alpha,-1} =e_{\al,-1}^\ww +\sum_{\beta\in \bar{\Delta}_+
\atop \beta>\alpha} Q_{\beta}^{\alpha} e_{\be,-1}^\ww
\label{eq:linear-independent}
\end{align}
for some polynomials $Q_{\beta}^{\alpha}\in A$.

Since $(\iota(x_{-1}) \vac)_{(m)} =x^p_{m} -(x^{[p]})_{mp}$ is
central in $U(\affg)$ and $\ww$ is a $\affg$-homomorphism,
$(\iota(x_{-1})^\ww \vac)_{(m)}$ commutes with $e_{\g,n}^\ww$ for any $n$ and any $\g \in \bar
\Delta_+$ (in particular for $\g \ge \al$). By Lemma~\ref{lem:comm
a*}, $(\iota(x_{-1})^\ww \vac)_{(m)}$ also commutes with
$(Q_{\beta}^{\alpha})_{(n)}$ for any $n$. Now by applying
$Y(-,z)_{(n)}$ to \eqref{eq:linear-independent},
$(\iota(x_{-1})^\ww \vac)_{(m)}$ commutes with $a_{\alpha,n}$.
\end{proof}

\begin{remark}
It is elementary to show by induction on $k$ that for any
$k,n,m,i,\be$ one has
$$
[(h_{i,n}^\ww)^k,a_{\be,m}] = \sum_{d=1}^k \binom{k}{d} \be(h_i)
\cdot a_{\be,nd+m} (h_{i,n}^\ww)^{k-d}.
$$
It follows that $[(h_{i,n}^\ww)^p-h_{i,np}^\ww,a_{\be,m}] = 0.$
Similarly, $[(h_{i,n}^\ww)^p-h_{i,np}^\ww,a_{\be,m}^*] =0.$
\end{remark}

\begin{proposition}  \label{prop:eiota}
Formulas \eqref{eq:ep}-\eqref{eq:hp} hold.
\end{proposition}

\begin{proof}
Let us fix $i$. 
The strategy is similar to the proof of Lemma~\ref{pFormula}.
We extend the notation to write $P_{\al_i}^i \big(a_\al^*(z)\big) =a_{\al_i}(z)$.  
Then by \eqref{eq:iotaehf} and \eqref{eq:e}, we can write 
\begin{equation}  \label{eq:eizp}
(e_i(z)^p)^\ww
=\sum_{\be \in \bar{\Delta}_+} P_\be^i \big(a_\al^*(z)\big)^p a_\be(z)^p+\sum_{t=1}^{ p-1} \no Y_t(a_\al^*(z)) a_\be(z)^t\no,
\end{equation}
for some differential polynomials $Y_t$, where the last summand arises from contractions in Wick's formula. 
Lemma~\ref{lem:comm a*} can be rephrased by saying that
$(e_i(z)^p)^\ww$ commutes with $a_\be^*(z)$. The first summand on the right-hand side of \eqref{eq:eizp}
commute with  $a_\be^*(z)$, but $a_\be(z)^t$, for $1\le t \le p-1$, do not commute with $a_\be^*(z)$. 
Hence we must have $Y_t=0$ for all $t$ by a downward induction on $t$. This proves \eqref{eq:ep}.

Similarly, noting  $\be(h_i)$ is integral and using \eqref{eq:iotaehf} and \eqref{eq:h}, we have
\begin{equation}  \label{eq:hizp}
(h_i(z)^p -h_i(z))^\ww
= -\sum_{ \be \in \bar{\Delta}_+}  \beta(h_i) \big[(\no a_\be^*(z) a_\be(z)\no )^p -\no a_\be^*(z) a_\be(z)\no \big] + b_i(z)^p
-\der^{(p-1)} b_i(z).
\end{equation}
Now write $(\no a_\be^*(z) a_\be(z)\no )^p -\no a_\be^*(z) a_\be(z)\no
= a_\be^*(z)^p a_\be(z)^p +\sum_{t=1}^{ p-1} \no X_t(a_\be^*(z)) a_\be(z)^t\no$, for some differential
polynomials $X_t$ (Here $\beta$ is fixed).  But by considering the commutation of \eqref{eq:hizp} with $a_\be^*(z)$
and applying Lemma~\ref{lem:comm a*}, we conclude that $X_t=0$ for each $t$. This proves \eqref{eq:hp}.
\end{proof}


To complete the proof of Theorem~\ref{th:ehfww} it remains to prove \eqref{eq:fp}.
Denote by $\fing_\Z$ the $\Z$-lattice generated by the Chevalley generators
$e_\al, f_\al$ and $h_i$, for $\al \in \bar{\Delta}^+$ and $i=1,\ldots, \ell$. 
Denote by $V_\Z$ the $\Z$-lattice of $\Vgc$
spanned by all possible $a_{-i_1} b_{-i_2} c_{-i_3}  \ldots \vac$,
where $a,b,c\ldots \in \fing_\Z$ and $i_1, i_2, i_3,\ldots \ge 1$.
Writing a general vertex operator $Y(a,z)= \sum_{n\in \Z} a_{(n)} z^{-n-1}$,
we recall a general formula from the theory of vertex algebras  (cf. \cite{Fr2}):
\begin{equation} \label{eq:xymn}
[x_{(m)}, y_{(n)}] =\sum_{i \ge 0} \binom{m}{i} \big(x_{(i)} y\big)_{(m+n-i)},
\qquad i \ge 0.
\end{equation}

>From a similar consideration as in the proof of Proposition~\ref{prop:eiota} above, we conclude that 
$\iota(f_i (z))^\ww$ is of the form
\begin{align}
\iota(f_i (z))^\ww = & \sum_{\be \in \bar{\Delta}_+} Q_\be^i \big(a_\al^*(z)\big)^p a_\be(z)^p
 +  \eta \big( \partial_z a_{\al_i}^*(z) \big)^p
 +  a_{\al_i}^*(z)^p R(b_i(z)),
\end{align}
where $\eta \in \K$ and 
\begin{equation}  \label{eq:bp}
R(b_i(z)) =Y(r_i, z) = b_i(z)^p +\ldots
\end{equation}
is a (normal ordered) polynomial in $b_i(z)$ and its derivatives and $r_i \in \pi^{\kappa -\kappa_c}$. 
Formula \eqref{eq:fp} now follows from the proposition below.

\begin{proposition}
\label{prop:exact}
We have
\begin{enumerate}
\item
$R(b_i(z)) =  b_i(z)^p -\der^{(p-1)} b_i(z)$;

\item
$\eta =(\kappa^p -\kappa) \l e_i, f_i \r.$
\end{enumerate}
\end{proposition}

\begin{proof}[Sketch of a proof]
It is possible to realize Wakimoto modules over $\Z$ as limits of twisting Verma modules, denoted by $\WZ$,
on which $e_{i,n}, h_{i,n}, f_{i,n}$ act. Then the formulas \eqref{eq:e}-\eqref{eq:f} are understood
as a congruence equation modulo $p \WZ$ when acting on any  $v \in \WZ$; 
moreover $(e^p_{-1} \vac)^\ww_{(n)} v \in p \WZ, (f^p_{-1} \vac)^\ww_{(n)} v \in p \WZ$, thanks to Lemmas~\ref{lem:comm a*}
and \ref{lem:comm a}. 
>From weight consideration, we have 
$
(e^p_{-1}\vac)_{(p-1)} f^p_{-1} \vac \equiv -p (h_{-1}^p -h_{-p}) \mod p^2 V_\Z.
$

On the other hand, for $n\ge 0$ and $v\in \WZ$,  we have
\begin{align} \label{eq:efpn}
\Big((  e^p_{-1}\vac)_{(p-1)} f^p_{-1} \vac \Big)^\ww_{(n)} v
   &
 =\sum_{i =0}^{p-1} \left( (-1)^i \binom{p-1}{i} (e^p_{-1}\vac )^\ww_{(p-1-i)} (f^p_{-1} \vac)^\ww_{(n+i)} v \right.
    \notag  \\
& \qquad   -(-1)^{p-1}  ( f^p_{-1}\vac )^\ww_{(n+p-1-i)} (e^p_{-1} \vac)^\ww_{(i)} v \Big).
\end{align}
But if we compute \eqref{eq:efpn} by applying \eqref{eq:e}-\eqref{eq:f}, 
the only term involving $b_{i,n}$ is given by $-r_i v$, which by \eqref{eq:bp} must be equal to $-(b_{i,-1}^p -b_{i,-p})v$ modulo $p\WZ$.
Part (1) now follows from this together with \eqref{eq:bp}.

Part (2) reduces to the $\mf sl_2$ case by Lemma~\ref{lem:Qi}.
\end{proof}

\section{Irreducible baby Wakimoto modules $\mf w (- \rho)$}
\label{sec:irred}

\subsection{ Mathieu's character formula reformulated}

For an integral weight $\la \in \affh^*$, denote by $\mf l (\la)$ the irreducible quotient $\affg$-module  of the Verma $\affg$-module 
of high weight $\la$. Recall the torus $T$ from \$\ref{sec:affinep}. 
Then $\mf l(\la)$ is naturally an $\affg$-$T$-module in the sense of Jantzen \cite{Jan}, and this allows one
to makes sense its (formal) character $\text{ch}\, \mf l (\la)$ in the usual sense. 

Mathieu \cite{Ma} proved the following character formula
\begin{equation}  \label{eq:char}
\text{ch } \mf l (- \rho) 
= e^{-\rho} \prod_{\alpha \in \Delta_+^{\rm re}} \frac{(1-e^{- p\al})}{(1-e^{-\al})}.
\end{equation}
Note that $-\rho$ is a weight at the critical level $\kappa_c$. 
We have the following reformulation of a main result of Mathieu, which has the 
advantage that the irreducible $\affg$-module  $\mf l(-\rho)$ is realized explicitly as the baby Wakimoto module $\mf w(-\rho)$
in terms of (restricted) free fields.

 \begin{theorem}  \label{th:rho}
 The baby Wakimoto module $\mf w(-\rho)$ is the irreducible high weight $\affg$-module of high weight $-\rho$.
 \end{theorem}
 
 \begin{proof}
 By construction of the baby Wakimoto module, we have the following character formula:
\begin{equation*}
\text{ch } \mf w (- \rho)  
= e^{-\rho} \prod_{\alpha \in \Delta_+^{\rm re}} \frac{(1-e^{- p\al})}{(1-e^{-\al})}.
\end{equation*}
The (obvious) surjective homomorphism $\mf w(-\rho) \rightarrow \mf l(-\rho)$ must be an
isomorphism by a character comparison. 
 \end{proof}

As modules over $\affg$, we have $\mf l(-\rho)=\mf l((p-1)\rho)$. 
(More general $\mf l (\la) =\mf l(\mu)$ if $\la -\mu$ is a $p$-multiple of an integral weight of $\affg$.)
Denote by $U=\K \otimes_\Z U_\Z$, where $U_\Z$
is the Kostant-Garland $\Z$-form of the universal enveloping algebra of $\affg$. 
Denote by $L(\la)$ (the notation $l(\la)$ was used in \cite{Ma}) 
the irreducible highest weight $U$-module of highest weight $\la$ (which is assumed to be integral).
Note that the restricted enveloping algebra is a subalgebra of $U$, i.e., ${\mf u}_0(\affg) \subseteq U$. 
Since $(p-1)\rho$ is a restricted weight, it follows by Mathieu \cite[Lemma~1.7]{Ma} that $L((p-1)\rho)$
when restricted to $\mf u_0(\affg)$ remains to be irreducible, and hence $L((p-1)\rho) \cong \mf l((p-1)\rho)$ 
as $\affg$-modules. Therefore Theorem~\ref{th:rho} and \eqref{eq:char} have
the following implication. 

\begin{corollary} [\cite{Ma}]
We have the following character formulas:
\begin{align} 
 \label{eq:p-1rho}
{\rm ch }\;  l( (p-1) \rho) &= e^{ (p-1)\rho}   \prod_{\alpha \in \Delta_+^{\rm re}} \frac{(1-e^{- p\al})}{(1-e^{-\al})}, 
 \\
 \label{eq:rho}
{\rm ch }\; l(- \rho) &= e^{-\rho}\prod_{\alpha \in \Delta_+^{\rm re}} \frac{1}{1-e^{-\al}}.
\end{align}
\end{corollary}
The above two formulas are equivalent by Steinberg tensor product theorem.

\subsection{Conjectures and further problems}

Recall the vertex algebra $V^{\kappa} (\affg_\C)$ over $\C$ has trivial center at a non-critical level $\kappa$;
at the critical level, $V^{\kappa_c} (\affg_\C)$ has a large center, which is explicitly described in \cite{Fr1, Fr2}.
This center continues to make sense for  $V^{\kappa_c} (\affg)$ over $\K$ in characteristic $p$; we shall refer to 
this as the Harish-Chandra center of $V^{\kappa_c} (\affg)$ and denote it by 
$\zz_{\rm HC}(V^{\kappa_c} (\affg))$.

\begin{conjecture}
\begin{enumerate}
\item
For $\kappa \neq \kappa_c$, the center of the vertex algebra $\Vg$ coincides with the $p$-center $\zz_0(\Vg)$.

\item
The center of the vertex algebra $V^{\kappa_c} (\affg)$ is generated by the Harish-Chandra center 
$\zz_{\rm HC}(V^{\kappa_c} (\affg))$ and the $p$-center $\zz_0(V^{\kappa_c} (\affg))$.
\end{enumerate}
\end{conjecture}

A $p$-character $\xi^M$ of  $\mathcal A^\affg$ is called graded if $\xi^M(a_{\al,n})=0 =\xi^M(a^*_{\al,n})$ for all $n\neq 0$. 
A $p$-character $\xi^\pi$ of $\B^\h_\kappa$ is  graded if $\xi^\pi(b_{i,n})=0$ for all $n\neq 0$ and $1\le i \le \ell$. 
Similarly, a graded $p$-character for $\affg$ can be defined. 

The modular representation theory of (finite-dimensional) Lie algebras has been well developed; cf. the review of 
Jantzen \cite{Jan}. It will be of great interest to develop modular representation theory for an affine Lie algebra $\affg$, say when
the $p$-character is (graded) semisimple or nilpotent. 
In particular, one may ask if the baby Wakimoto modules are irreducible for {\em generic}  (graded) semisimple $p$-characters.
The modular representation theory of the algebra $U$ (or the corresponding algebraic group of $\affg$) has been very
challenging; we refer to \cite{Lai} and the references therein for results in this direction. 
The modular representation theory of $\affg$ should be somewhat more accessible and flexible by imposing
various conditions on $p$-characters.


\end{document}